\documentclass[12pt]{amsart}
\usepackage{amsmath,amssymb,amsfonts,enumerate,amsthm,amscd,color, graphics}
\usepackage[bookmarksnumbered, colorlinks, plainpages]{hyperref}
\usepackage[all]{xy}

\numberwithin{equation}{section}

\newtheorem{thm}{Theorem}[section]

\newtheorem{cor}[thm]{Corollary}
\newtheorem{prop}[thm]{Proposition}
\newtheorem{rem}[thm]{Remark}

\newtheorem{defsrem}[thm]{Definitions and Remark}

\newcommand{\fa}{\mathfrak{a}}
\newcommand{\fb}{\mathfrak{b}}
\newcommand{\fc}{\mathfrak{c}}

\newcommand{\m}{\mathfrak{m}}

\newcommand{\q}{\mathfrak{q}}
\newcommand{\N}{\mathbb{N}}
\newcommand{\Z}{\mathbb{Z}}

\newcommand{\cd}{\operatorname{cd}}
\newcommand{\grade}{\operatorname{grade}}

\newcommand{\faR}{f_{\fa}^{R_+}}
\newcommand{\fR}{f_{R_+}}
\newcommand{\fg}{finitely generated}
\newcommand{\fgR}{finitely generated $R$-module}

\newcommand{\fggR}{finitely generated graded $R$-module}
\newcommand{\fG}{finitely graded}

\newcommand{\lM}{linked by $I$ over $M$}

\newcommand{\R}{$R=\bigoplus_{n\in \N_0}R_n$}

\newcommand{\Ro}{$(R_0,\m_0)$}
\newcommand{\rs}{regular sequence}
\newcommand{\seq}{sequence}
\newcommand{\ab}{$\fa \stackrel{h} \sim_{(I;M)} \fb$}
\newcommand{\aR}{$\fa \stackrel{h} \sim_{(I;M)} R_+$}
\newcommand{\bM}{$\fb M=IM:_M \fa$}
\newcommand{\aM}{$\fa M=IM:_M \fb$}
\newcommand{\IM}{$\fa M \cap \fb M = IM$}
\newcommand{\homo}{homogeneous}
\newcommand{\st}{such that}

\newcommand{\MV}{Mayer-Vietoris}

\newcommand{\rCM}{relative Cohen-Macaulay}

\begin{document}

\address{ }
\email{}

\author{Maryam Jahangiri}
\address{Department of Mathematics, Faculty of Mathematical Sciences and Computer, Kharazmi University, Tehran,
Iran.}
\email{jahangiri@khu.ac.ir, jahangiri.maryam@gmail.com }
 \thanks{$^*$Corresponding author}

\author{Azadeh Nadali}
\email{a.nadali1984@gmail.com}

\author{Khadijeh Sayyari$^*$}
\email{std\underline{ }sayyari@khu.ac.ir, sayyarikh@gmail.com}


\title[  ASYMPTOTIC BEHAVIOUR OF GRADED LOCAL COHOMOLOGY MODULES VIA LINKAGE  ]
{ ASYMPTOTIC BEHAVIOUR OF GRADED LOCAL COHOMOLOGY MODULES VIA LINKAGE }

\begin{abstract}
Assume that $R=\oplus_{n\in \mathbb{N}_0}R_n$ is a standard graded algebra over the local ring $(R_0,\mathfrak{m}_0)$, $\mathfrak{a}$ is a homogeneous ideal of $R$, $M$ is a finitely generated graded
$R$-module and $R_+:=\oplus_{n\in \N}R_n$ denotes the irrelevant ideal of $R$.
 In this paper, we study the asymptotic behaviour of the set
$\{ \operatorname{grade}(\mathfrak{a} \cap R_0, H^{\operatorname{grade}(R_+,M)}_{R_+}(M)_n) \}_{n \in \mathbb{Z}}$  as
$n \rightarrow -\infty$, in the case where $\mathfrak{a}$ and $R_+$ are homogenously linked over $M$.

\end{abstract}

\maketitle
\section{introduction}

Let $R= \bigoplus_{n\in \N_0}R_n$ be a positively graded commutative Noetherian ring, $R_+=
\bigoplus_{n\in \N}R_n$ be the irrelevant ideal of $R$ and  $M$  denotes a
 graded $R$-module. Also, assume that $\fa$ is a \homo ~ ideal of $R$.
 Then for all $i \in \N_0$, the set of non-negative integers, the i-th local cohomology module $H^i_{\fa}(M)$ of $M$ with respect to $\fa$ has a natural grading (our terminology on local cohomology comes from \cite{BSH12}).
 Also, in the case where $M$ is \fg, it is well-known that the graded components of local cohomology modules $H^{i}_{R_+}(M)_n$ of $M$ with respect to the irrelevant ideal are trivial for sufficiently large values of $n$ and they are finitely generated as $R_0$-module for all $n \in \Z$, the set of integers (\cite[16.1.5]{BSH12}).
 The asymptotic behaviour of the components $H^i_{R_+}(M)_n$ when $n \rightarrow -\infty$ is more complicated and has been studied by many authors, too. See for example \cite{brodmann(survey)}, \cite{BH}, \cite{jzh} and \cite{thomasmarley}.

Assume that $R$ is  standard graded, i.e. $R_0$ is a commutative Noetehrian ring and $R$ is generated, as an
$R_0$-algebra, by finitely many elements of degree one, and $M$ is finitely generated. In \cite{jzh} the authors considered the problem of asymptotic behaviour of the set $\{\grade(\fa_0, H^i_{R_+}(M)_n)\}_{n \in \Z}$  when $n \rightarrow -\infty$, in the case where $(R_0,\m_0)$ is local and $\fa_0$ is a proper ideal of $R_0$.
They showed that this set is asymptotically stable in some special cases.
In this paper, we consider this problem, via linkage.

Following \cite{gradedlcm_JNS}, two \homo ~ ideals $\fa$ and $\fb$ are said to be \homo ly linked (or h-linked) by $I$ over $M$, denoted by \ab, if $I$ is generated by a \homo ~ $M$-\rs ~ and \aM ~ and \bM.
This is a generalization of the classical concept of linkage of ideals introduced by Peskine and Szpiro \cite{peskine1974liaison}.

Assume that the \homo ~ ideal $\fa$ and $R_+$ are h-linked over $M$. We show that the set
$\{\grade(\fa \cap R_0, H^{\grade(R_+,M)}_{R_+}(M)_n)\}_{n \in \Z}$ is asymptotically stable when $n \rightarrow -\infty$,
and also we determine its stable value, in some cases, see Proposition \ref{grade fa}. Also, in general case, we find a lower bound for this set.
More precisely, let
\[
f_{\fa }^{R_+}(M):=inf \{i | ~ R_+ \nsubseteq \sqrt{0:H^i_{\fa}(M)} \}
\]
denotes the $R_+$-finiteness dimension of $M$ relative to $\fa$ and $f_{R_+}(M):=f_{R_+}^{R_+}(M)$ be the finiteness dimension of $M$ relative to $R_+$. Then, we show that if \aR ~and $\grade(R_+,M)=f_{R_+}(M)$ then
$\grade(\fa \cap R_0, H^{\grade(R_+,M)}_{R_+}(M)_n) \geq f_{\fa +R_+}^{R_+}(M)-f_{R_+}(M)$ for all $n \ll 0$
 (Theorem \ref{grade a0}). As a corollary, we show that the graded components $H^{\grade(R_+,M)}_{R_+}(M)_n$ are \rCM ~ with respect to $\fa \cap R_0$ of degree $f_{\fa +R_+}^{R_+}(M)-f_{R_+}(M)$, in some certain cases (Corollary \ref{grade a 1}).
Recall that a non-zero \fgR ~ $X$ is said to be \rCM ~with respect to the ideal $\fc$ of degree $n$ if $H^i_{\fc}(X)=0$ for all $i\neq n$.\\
At the end, we show that if \Ro ~is local then the local cohomology modules $H^i_{\m_0R}(H^j_{R_+}(M))$ are Artinian,
for all $i,j \in \N_0$, when $R_+$ is linked by an $\m_0$-primary ideal (Proposition \ref{dim}).

Throughout the paper, \R ~ is a standard graded ring over the local base ring \Ro , $M$ is a \fggR ~ with $M\neq \Gamma_{R_+}(M)$.
Also, $\fa$ denotes a proper \homo ~ ideal of $R$ and we set $\fa_0:=\fa \cap R_0$.

\section{The results}

First, we recall the concept of homogeneously linked ideals from \cite{gradedlcm_JNS}, as well as some basic definitions which will be used later in the paper.

\begin{defsrem}
Let $\fb$ be a second \homo ~ ideal of $R$.
\begin{itemize}
  \item [(i)] Assume that $\fa M\neq M \neq \fb M$ and $I\subseteq \fa \cap \fb$ be an ideal generated by a \homo ~ $M$-\rs. Then we say that the ideals $\fa$ and $\fb$ are \homo ly linked (or h-linked) by $I$ over $M$, denoted \ab, if \bM ~and \aM. The ideals $\fa$ and $\fb$ are said to be  geometrically h-\lM ~if, in addition, \IM. Also, we say that the ideal $\fa$ is h-linked over $M$ if there exist \homo ~ ideals $\fb$ and $I$ of $R$ such that \ab. $\fa$ is h-$M$-selflinked by $I$ if $\fa \stackrel{h}\sim_{(I;M)} \fa$.
   \item [(ii)] Following \cite[9.1.5]{BSH12}, the $\fb$-finiteness dimension of $M$ relative to $\fa$ is defined to be
\[
f_{\fa}^{\fb}(M):= inf \{ i \in \N_0 |~ \fb \nsubseteq \sqrt{0:H^i_{\fa}(M)} \}.
\]
In the case where $\fb=\fa$, in view of \cite[9.1.3]{BSH12}, this invariant equals the finiteness dimension of $M$ relative to $\fa$.
That is
\[
f_{\fa}(M):= inf \{ i \in \N |~ H^i_{\fa}(M) \text{ is not finitely generated} \}.
\]
\end{itemize}
\end{defsrem}

  In the next proposition, we want to study the stability of the set $\{\grade(\fa_0,H^{\fR(M)}_{R_+}(M)_n)\}_{n \in \Z}$, of integers where $n \rightarrow -\infty$. This problem has already considered in \cite{jzh}.
Note that, by \cite[6.2.7]{BSH12} and \cite[3.16(ii)]{gradedlcm_JNS},
$\grade(R_+,M)\leq f_{R_+}(M) \leq f_{\fa+R_+}^{R_+}(M)$.

\begin{prop}\label{grade fa}
   Let  \aR ~  and set $t:=\grade(R_+,M)$.
      If $\fR(M)\neq t$ or $f_{\fa +R_+}^{R_+} (M)=t$, then $\grade(\fa_0,H^{\fR (M)}_{R_+}(M)_n)=0$ for all $n\ll 0$.\
\end{prop}
\begin{proof}
    By the \homo ~\MV ~ \seq~ and \cite[2.2]{cohomological3}, we have the following \homo ~exact \seq ~ of $R$-modules
  \begin{equation*}
    \ldots \longrightarrow H_I^{i-1}(M) \longrightarrow H_{\fa + R_+} ^i(M) \longrightarrow H_{\fa}^i(M) \oplus H_{R_+}^i(M) \longrightarrow H_I^i(M)  \longrightarrow \ldots ~.
  \end{equation*}
  It yields
  \begin{equation}\label{iso2}
    H_{\fa + R_+} ^i(M)\cong H_{\fa}^i(M) \oplus H_{R_+}^i(M) \qquad \text{for all $i \gneq t+1$} \tag{2.2}
  \end{equation}
  and, by \cite[6.2.7]{BSH12}, the exact \seq

  \begin{equation}\label{exact}
    0 \longrightarrow H_{\fa + R_+} ^t(M) \longrightarrow H_{\fa}^t(M) \oplus H_{R_+}^t(M) \longrightarrow H_I^t(M)
    \longrightarrow H_{\fa +R_+} ^{t+1}(M) \longrightarrow  H_{\fa}^{t+1}(M) \oplus H_{R_+}^{t+1}(M) \longrightarrow 0. \tag{2.3}
  \end{equation}

  These, result that $H^i_{\fa}(M)$ and $H^i_{R_+}(M)$ are respectively $R_+$-torsion and $\fa$-torsion for all $i\neq t$. So, if $\fR(M)\neq t$, then $\Gamma_{\fa}(H^{\fR(M)}_{R_+}(M))=H^{\fR(M)}_{R_+}(M)$. Therefore, by \cite[14.1.12]{BSH12} and
  \cite[4.8]{brodmann(survey)},
  \[
  \Gamma_{\fa_0}(H^{\fR(M)}_{R_+}(M)_n)=H^{\fR(M)}_{R_+}(M)_n \neq 0 \quad \text{ for all $n \ll 0$}.
   \]

   This, in view of \cite[6.2.7]{BSH12}, prove the claim. \\
      Now, assume that $f_{\fa +R_+}^{R_+}(M)=t$. By \cite[3.4]{nagel1994cohomological},
      \[
      \Gamma_{\fa}(H^{t}_{R_+}(M))\cong \Gamma_{\fa}(\Gamma_{R_+}(H^{t}_{I}(M))) = \Gamma_{\fa+R_+}(H^{t}_{I}(M))\cong
      H^{t}_{\fa +R_+}(M).
      \]
      So, using \cite[2.2]{jzh} and \cite[14.1.12]{BSH12}, $\Gamma_{\fa_0}(H^{t}_{R_+}(M)_n) \neq 0$ for all $n\ll 0$.
       This implies that $\grade(\fa_0,H^{\fR (M)}_{R_+}(M)_n)=0$ for all $n\ll 0$, as desired.
\end{proof}

 Let  $N=\bigoplus_{n \in \Z}N_n$ be a graded $R$-module. Then following \cite{thomasmarley},
 $N$ is called finitely graded if  $N_n=0$  for all but finitely many $n \in \Z$. Also, we set
 \[
       g_{\fa}(N) := sup\{k \in \N_0 | H_{\fa}^i(N) ~\text {is finitely graded for all}~ i < k  \}.
       \]

  \begin{rem}\label{rem}
      Note that, by \cite[2.3]{thomasmarley}, if $N$ is \fg , then $g_{\fa}(N)=\faR (N)$.
  \end{rem}

In view of \cite[3.17]{gradedlcm_JNS}, if \aR ~ and  one of the conditions of proposition \ref{grade fa} holds, then $\grade(\fa_0,H^{\fR(M)}_{R_+}(M)_n)=f_{\fa +R_+}^{R_+} (M)-\fR(M)$ for all $n\ll 0$. In the next theorem and corollary, we consider a case where
\aR ~ and none of the conditions of \ref{grade fa} establish. We show that $\grade(\fa_0,H^{\fR(M)}_{R_+}(M)_n)$ is bounded bellow for all $n \ll 0$ and also, in a special case, it is asymptotically stable, as $n \rightarrow -\infty$.

\begin{thm}\label{grade a0}
  Let \aR ~ and assume that $\grade(R_+,M)=\fR(M) < f_{\fa +R_+}^{R_+} (M)$. Then
   $\grade(\fa_0,H^{\fR(M)}_{R_+}(M)_n) \geq f_{\fa +R_+}^{R_+} (M)-\fR(M)$ for all $n\ll 0$.
\end{thm}

\begin{proof}
  Set $t:=\grade(R_+,M)$. Note that, in view of \cite[3.16(ii)]{gradedlcm_JNS}, $f_{\fa +R_+}^{R_+}(M)\geq t+1$. Now, consider the following convergence of spectral \seq s
  \cite[11.38]{rotman-spectral},
  \[
  E_2^{i,j}=H^i_{\fa_0R}(H^j_{R_+}(M)) \stackrel{i} \Rightarrow H^{i+j}_{\fa +R_+}(M).
  \]
  By \eqref{iso2} and \eqref{exact}, $H^j_{R_+}(M)$ is $\fa$-torsion for all $j\neq t$. As a result,
   \begin{equation}\label{Er1}
     E_2^{i,j}=0 \quad \qquad \text{ for all $i \geq 1$ and all $j\neq t$},\tag{2.4}
     \end{equation}
      and the above convergence of spectral \seq s yields
  \begin{equation}\label{Er2}
  \Gamma_{\fa_0}(H^{t+i}_{R_+}(M)_n)=(E_2^{0, t+i})_n= \ldots =(E_{i+1}^{0, t+i})_n \supseteq \ker(d_{i+1}^{0, t+i})_n=
  (E_{i+2}^{0, t+i})_n = (E_\infty^{0, t+i})_n, \tag{2.5}
  \end{equation}
  for all $i \in \N$ and all $ n \in \Z$. Therefore, if $1 \leq i \lneq f_{\fa +R_+}^{R_+}(M)-t$ then, by \eqref{iso2}, \eqref{exact} and \ref{rem}, $H^{i+t}_{R_+}(M)$ is \fG ~ and \eqref{Er2} implies that
  \[
  (E_2^{0, t+i})_n = (E_\infty^{0, t+i})_n=0  \qquad \text{for all $1 \leq i \lneq f_{\fa +R_+}^{R_+}(M)-t$ and all $n\ll 0$}.
  \]
   Also, using the concept of convergence of spectral \seq s and \ref{rem}, we have
  \[
  0=H^{i+t}_{\fa +R_+}(M)_n\cong (E_\infty ^{i,t})_n=(E_2^{i,t})_n=H^i_{\fa_0}(H^t_{R_+}(M)_n),
  \]
  for all $1 \leq i \lneq f_{\fa +R_+}^{R_+}(M)-t$ and all $n\ll 0$. In addition, in view of \eqref{Er1} and the assumption on $t$,
  \[
  0=H^t_{\fa+R_+}(M)_n \cong (E^{0,t}_{\infty})_n = (E^{0,t}_{2})_n = \Gamma_{\fa_0}(H^t_{R_+}(M)_n), \qquad
  \text{ for all $n\ll 0$}.
  \]
   Therefore, by \cite[6.2.7]{BSH12},
   \[
   \grade(\fa_0,H^{t}_{R_+}(M)_n) \geq f_{\fa +R_+}^{R_+} (M)-t \quad \qquad \text{ for all $n\ll 0$}.
   \]

\end{proof}

Let $N$ be an $R$-module. The cohomological dimension of $N$ with respect to $\fa$ is defined to be
\[
\cd(\fa,N):=sup \{i \in \Z | H^i_{\fa}(N)\neq 0 \}.
\]

A \fgR ~$X$ is called \rCM~ with respect to the ideal $\fb$ of degree $n$ if $H^i_{\fb}(X)=0$ for all $i\neq n$.\\
In the next item, using the above theorem, we show that the graded components $H^{\grade(R_+,M)}_{R_+}(M)_n$ are \rCM ~ $R_0$-modules with respect to $\fa_0$, in some special cases.

\begin{cor}\label{grade a 1}
  Let the situations be as in theorem \ref{grade a0} and set $t:= \grade(R_+,M)$. Then in each of  the following cases,
  $H^t_{R_+}(M)_n$ is a \rCM ~ $R_0$-module with respect to $\fa_0$ of degree $ f_{\fa +R_+}^{R_+}(M)-t$ for all $n \ll 0$, i.e.
  $\grade(\fa_0,H^{t}_{R_+}(M)_n)= \cd (\fa_0,H^{t}_{R_+}(M)_n)= f_{\fa +R_+}^{R_+}(M)-t$ for all $n\ll 0$.
  \begin{itemize}
    \item[(i)]  $M$ is \rCM ~ with respect to $\fa +R_+$.
    \item[(ii)]  $\fa$ is generated by \homo ~ elements of degree 0. 
    \item [(iii)]  $\cd(\fa ,M) \in \{t, t+1\}$.
  \end{itemize}
\end{cor}
\begin{proof}
  \begin{itemize}
    \item [(i)] By hypothesis and \cite[3.15(ii)]{gradedlcm_JNS}, $\cd(\fa+R_+ ,M)=f_{\fa +R_+}^{R_+}(M)$.
     Since $H^i_I(M)=0$ for all $i\neq t$, using a \homo ~ \MV ~ \seq ~and \cite[2.2]{cohomological3} we have the following \homo ~ exact \seq s of $R$-modules

  \begin{equation}\label{exact2}
  0 \longrightarrow H_{\fa + R_+} ^t(M) \longrightarrow H_{\fa}^t(M) \oplus H_{R_+}^t(M)  \longrightarrow X \longrightarrow 0, \tag{2.6}
 \end{equation}
 and
 \begin{equation}\label{exact3}
   0 \longrightarrow X \longrightarrow H_I^t(M)  \longrightarrow Y \longrightarrow 0 \tag{2.7},
 \end{equation}
 where $X$ is a submodule of $H^t_I(M)$ and $Y$ is a submodule of $H^{t+1}_{\fa+R_+}(M)$.\\
 Now, by affecting $\Gamma_{\fa}(-)$ on exact \seq s ~ \eqref{exact2} and \eqref{exact3} and using 
 \cite[3.4]{nagel1994cohomological}, the following isomorphisms have been obtained

 \begin{equation}\label{iso4}
   H^i_{\fa}(H^t_{R_+}(M))\cong H^i_{\fa}(H^t_{I}(M))\cong H^{i+t}_{\fa}(M),  \qquad \quad \text{for all}~ i > 1 \tag{2.8}.
 \end{equation}
     Also, using \cite[2.3]{cohomological3}, $\cd(\fa ,M)\leq \cd(\fa +R_+,M)=f_{\fa +R_+}^{R_+}(M)$. Therefore,
     \begin{equation}\label{zero}
          H^i_{\fa}(H^t_{R_+}(M))=0 \qquad \quad \text{ for all $i\gneq f_{\fa +R_+}^{R_+}(M)-t $}.\tag{2.9}
    \end{equation}
   On the other hand, by  \cite[14.1.12]{BSH12} and the fact that $H^t_{R_+}(M)$ is $R_+$-torsion, we have
   \begin{equation}\label{iso5}
   H^i_{\fa}(H^t_{R_+}(M))_n \cong H^i_{\fa_0+R_+}(H^t_{R_+}(M))_n \cong H^i_{\fa_0R}(H^t_{R_+}(M))_n \cong
   H^i_{\fa_0}(H^t_{R_+}(M)_n) \tag{2.10}
   \end{equation}
    for all $i$ and all $n $.
        Therefore, by  \eqref{zero},
        \[
        \grade(\fa_0,H^{t}_{R_+}(M)_n)\leq \cd(\fa_0,H^{t}_{R_+}(M)_n) \leq f_{\fa +R_+}^{R_+}(M)-t, \quad
        \text{for all $n \in \Z$}.
        \]
        Now, using \ref{grade a0}, the result follows.
   \item [(ii)] Since $M$ is \fg , by \eqref{iso4}, \cite[14.1.12]{BSH12} and \cite[Theorem 1]{kirby1973artinian},
    \[
    H^i_{\fa_0}(H^t_{R_+}(M)_n) \cong H^{i+t}_{\fa_0}(M_n)=0 \qquad \text{ for all $i > 1$ and all $n\ll 0$}.
    \]
    So, $\grade(\fa_0,H^{t}_{R_+}(M)_n) \leq \cd(\fa_0,H^{t}_{R_+}(M)_n) \leq 1$ for all $n\ll 0$. Also, by \ref{grade a0}, $\grade(\fa_0,H^{t}_{R_+}(M)_n)\geq f_{\fa +R_+}^{R_+}(M)-t \geq 1$, for all $n \ll 0$. This proves the claim.
    \item [(iii)] By hypothesis and \eqref{iso4}, $H^i_{\fa}(H^t_{R_+}(M))=0$ for all $i > 1$ and the statement holds as part (ii).

  \end{itemize}
\end{proof}

  The cohomological finite length dimension of $M$ is defined to be
  \[
 g(M):= \inf \{i\in \N \mid l_{R_0}(H^i_{R_+}(M)_n) = \infty ~ \text{for infinitely many n} \},
 \]
 where $ l_{R_0}(H^i_{R_+}(M)_n)$ is the length of the $R_0$-module $H^i_{R_+}(M)_n$.

The following proposition considers Artinianness of some local cohomology modules with respect to linked ideals as well as the stability of a set of integers.

\begin{prop}\label{dim}
  Assume that $\q_0$ is an $\m_0$-primary ideal of $R_0$ such that  $\q_0R \stackrel{h}\sim_{(I;M)} R_+$ and set $t:=\grade(R_+,M)$. Then the following statements hold.
  \begin{itemize}
    \item [(i)] The set $\{\dim_{R_0} M_n\}_{n \in \Z}$ of integers is asymptotically stable for $n\rightarrow +\infty$  and it equals $t$, i.e.
    $\dim_{R_0} M_n=t$ for all $n \gg 0$.
    \item [(ii)] $H^j_{R_+}(M)$ and $H^j_{\m_0R}(M)$ are Artinian $R$-modules for all $j\neq t$. Also, if $g(M) <\infty$, then
    $H^i_{\m_0R}(H^j_{R_+}(M))$ is Artinian for all $i$ and all $j$  and that $g(M)=t$.
  \end{itemize}
\end{prop}
\begin{proof}
\begin{itemize}
  \item [(i)]
  Since $M$ is finitely generated, by \cite[Theorem 1]{kirby1973artinian}, $R_1M_n=M_{n+1}$ for all $n\gg 0$. It shows that
   $\dim_{R_0} M_n \geq \dim_{R_0} M_{n+1}$ for
  all $n\gg 0$. Thus, there exists $l \in \N$ \st
  \begin{equation*}
    \dim_{R_0} M_n=l \qquad \qquad \quad \text{for all} ~n\gg 0.
  \end{equation*}
  On the other hand, by  \cite[14.1.12]{BSH12} and \cite[3.9]{gradedlcm_JNS},
  \[
  H^{i}_{\m_0}(M_n)\cong H^i_{\q_0}(M_n) \cong H^i_{q_0R}(M)_n=0 \qquad \quad \text{for all} ~i\neq t ~\text{and all} ~ n\gg 0.
  \]
   So, using \cite[6.1.4]{BSH12}, $l=t$.

   \item [(ii)] The first statement holds by \cite[7.1.3]{BSH12},  \eqref{iso2} and \eqref{exact}.  \eqref{iso2} and \eqref{exact}, also, ensure that $H^j_{R_+}(M)$ is $\m_0R$-torsion for all $j\neq t$, so $H^i_{\m_0R}(H^j_{R_+}(M))$ is Artinian for all $i$ and all $j\neq t$.\\
        For the case $j=t$, according to the definition of $g(M)$, $t \leq g(M)$. So, by \cite[2.4]{jzh}, $H^i_{\m_0R}(H^t_{R_+}(M))$ is Artinian for $i=0,1$ and also it does by \cite[1.2.3]{BSH12} and \eqref{iso4} for all $i>1$.\\
   Moreover, since $H^j_{R_+}(M)$ is Artinian for all $j\neq t$, by \cite[Theorem 1]{kirby1973artinian}, $H^j_{R_+}(M)_n$
   is an Artinian $R_0$-module for all $j\neq t$ and all $n \in \Z$. Therefore, using \cite[16.1.5]{BSH12},
   $ l_{R_0}(H^j_{R_+}(M)_n) < \infty$ for all $j\neq t$. As a result $g(M)=t$.
\end{itemize}
\end{proof}


\begin{thebibliography}{99}

\bibitem{brodmann(survey)} M. Brodmann. Asymptotic behaviour of cohomology: Tameness, supports and associated primes. In
S. Ghorpade, H. Srinivasan, and J. Verma, editors, Commutative Algebra and Algebraic Geometry,
volume 390, pages 31-61, 2005. Contemporary Mathematics.

\bibitem{BH}  M. Brodmann and M. Hellus. Cohomological patterns of coherent sheaves over projective schemes.
Journal of Pure and Applied Algebra, 172:165-182, 2002.

 \bibitem{BSH12}  M. P. Brodmann and R. Y. Sharp. Local cohomology: an algebraic introduction with geometric applications, volume 136 of Cambridge Studies in Advanced Mathematics. Cambridge university press, 2012.

 \bibitem{jzh} S. H. Hassanzadeh, M. Jahangiri, and H. Zakeri. Asymptotic behaviour and artinian property of graded
local cohomology modules. Communications in Algebra, 37(11):4095-4102, 2009.

 \bibitem{gradedlcm_JNS}  M. Jahangiri, A. Nadali, and Kh. Sayyari. Graded local cohomology modules with respect to the linked
ideals. arxiv: 2102.09015.


 \bibitem{cohomological3}  M. Jahangiri and Kh. Sayyari. Cohomological dimension with respect to the linked ideals. Journal of
Algebra and Its Applications. DOI: 10.1142/s0219498821501048.


 \bibitem{kirby1973artinian} D. Kirby. Artinian modules and hilbert polynomials. The Quarterly Journal of Mathematics, 24(1):47-57, 1973.


 \bibitem{thomasmarley} T. Marley. Finitely graded local cohomology and the depths of graded algebras. Proceedings of the
American Mathematical Society, 123(12):3601-3607, 1995.


 \bibitem{nagel1994cohomological} U. Nagel and P. Schenzel. Cohomological annihilators and castelnuovo-mumford regularity. Commutative Algebra, 159:307-328, 1994.


\bibitem{peskine1974liaison} C. Peskine and L. Szpiro. Liaison des vari{\'e}t{\'e}s alg{\'e}briques. I. Inventiones mathematicae, 26(4):271-302, 1974.


 \bibitem{rotman-spectral} J. J. Rotman. An introduction to homological algebra. Academic Press, 1979.



\end{thebibliography}

\end{document}